\documentclass[leqno]{eect}

\usepackage{amsmath, amsfonts, amssymb}
\usepackage{paralist}
\usepackage{graphics}
\usepackage{epsfig}
 \usepackage[colorlinks=true]{hyperref}
\hypersetup{urlcolor=blue, citecolor=red}

\def\currentvolume{x}
 \def\currentissue{x}
  \def\currentyear{2014}
   \def\currentmonth{xxx}
    \def\ppages{1--xx}
     \def\DOI{10.3934/eect.2013.2.xx}

\newtheorem{theorem}{Theorem}[section]

\newtheorem{lemma}[theorem]{Lemma}

\theoremstyle{definition}

\numberwithin{equation}{section}

\def\bfR{\mathbb{R}}
\def\bfN{\mathbb{N}}
\def\mcS{\mathcal{S}}
\def\mcT{\mathcal{T}}
\def\supp{\mbox{supp}}

\title[1D PLASMA MODEL WITH TRANSPORT FIELD]{A One-dimensional kinetic model of plasma dynamics with a transport field}

\author[Charles Nguyen, Jennifer Anderson, and Stephen Pankavich]{}


 \email{charles.nguyen@mavs.uta.edu} 
 \email{jra1116@math.tamu.edu}
 \email{pankavic@mines.edu}

 \thanks{This work was supported by the Center for Undergraduate Research in Mathematics under NSF Grant DMS-0636648 and independently by the National Science Foundation under NSF Grants DMS-0908413 and DMS-1211667.  Additionally, a portion of this work was submitted in partial fulfillment of the requirements of the second author's Master of Science degree at the University of Texas at Arlington.}

\begin{document}

\maketitle

\centerline{\scshape Charles Nguyen}
\medskip
{\footnotesize
 \centerline{Department of Mathematics}
 \centerline{University of Texas at Arlington}
 \centerline{Arlington, TX 76019 USA}

}

\medskip

\centerline{\scshape Jennifer Anderson}
\medskip
{\footnotesize
 \centerline{Department of Mathematics}
 \centerline{Texas A\&M University}
  \centerline{College Station, TX 77843 USA}

}

\medskip

\centerline{\scshape Stephen Pankavich}
\medskip
{\footnotesize
 \centerline{Department of Applied Mathematics and Statistics}
 \centerline{Colorado School of Mines}
 \centerline{Golden, CO 80002 USA}

} 

\medskip

\begin{abstract}
Motivated by the fundamental model of a collisionless plasma, the Vlasov-Maxwell (VM) system, we consider a related, nonlinear system of partial differential equations in one space and one momentum dimension.  As little is known regarding the regularity properties of solutions to the non-relativistic version of the (VM) equations, we study a simplified system which also lacks relativistic velocity corrections and prove local-in-time existence and uniqueness of classical solutions to the Cauchy problem.  For special choices of initial data, global-in-time existence of these solutions is also shown.  Finally, we provide an estimate which, independent of the initial data,  yields additional global-in-time regularity of the associated field.
\end{abstract}

\section{Introduction}
A plasma is a partially or completely ionized gas. Such a form of matter occurs if the velocity of individual particles in a material achieves an enormous magnitude, perhaps a sizable fraction of the speed of light.  Plasmas are widely used in solid state physics since they are great conductors of electricity due to their free-flowing abundance of ions and electrons. When a plasma is of low density or the time scales of interest are sufficiently small, it is deemed to be ``collisionless'', as collisions between particles become infrequent.  Many examples of collisionless plasmas occur in nature, including the solar wind, galactic nebulae, the Van Allen radiations belts, and comet tails.

The fundamental equations which describe the time evolution of a collisionless plasma are given by the Vlasov-Maxwell system:

\begin{equation} \label{VM} \tag{VM} \left \{ \begin{gathered} \partial_t f + v \cdot \nabla_x f + \left (E + v \times
B \right ) \cdot \nabla_v f = 0 \\  \rho(t,x) = \int f(t,x,v) \ dv, \quad j(t,x)= \int v  f(t,x,v) \ dv \\  \partial_t E = \nabla \times B - j, \qquad \nabla \cdot E = \rho \\  \partial_t B = - \nabla \times E,
\qquad \nabla \cdot B = 0. \\ \end{gathered} \right.
\end{equation}
Here, $f$ represents the density of (positively-charged) ions in the plasma, while $\rho$ and $j$ are the charge and current density, and $E$ and $B$ represent electric and magnetic fields generated by the charge and current.  The independent variables, $t > 0$ and $x,v \in \bfR^3$ represent time, position, and velocity, respectively, and physical constants, such as the charge and mass of particles, as well as, the speed of light, have been normalized to one.  In the presence of large velocities, relativistic corrections become important and the corresponding system to consider is the relativistic analogue of (\ref{VM}), denoted by (RVM) and constructed by replacing $v$ with $$\hat{v} = \frac{v}{\sqrt{1 + \vert v \vert^2}}$$ in the first equation of (\ref{VM}), called the Vlasov equation, and in the integrand of the current $j$.  General references on the kinetic equations of plasma dynamics, such as (VM) and (RVM), include \cite{Glassey} and \cite{VKF}.  

Over the past twenty-five years significant progress has been made in the analysis of (RVM), specifically, the global existence of weak solutions (which also holds for (VM); see \cite{DPL}) and the determination of sufficient conditions which ensure global existence of classical solutions (originally discovered in \cite{GlStr}, and later in \cite{KlStaf}, and \cite{BGP}) for the Cauchy problem.  Additionally, a wide array of information has been discovered regarding the electrostatic versions of both (\ref{VM}) and (RVM) - the Vlasov-Poisson and relativistic Vlasov-Poisson systems, respectively.  These models do not include magnetic effects within their formulation, and the electric field is given by an elliptic, rather than a hyperbolic equation. This simplification has led to a great deal of progress concerning the electrostatic systems, including theorems regarding global existence, stability, and long-time behavior of solutions; though a global existence theorem for classical solutions with arbitrary data in the relativistic case has remained elusive.  Independent of these advances, many of the most basic existence and regularity questions remain unsolved for (\ref{VM}).  The main difficulty which arises is the loss of strict hyperbolicity of the kinetic system due to the possibility that particle velocities $v$ may travel faster than the propagation of signals from the electric and magnetic fields, which do so at the speed of light $c = 1$. As one can see, this difficulty is remedied by the inclusion of relativistic velocity corrections which uniformly constrain velocities $\vert \hat{v} \vert < 1$.  In many physical systems one does not consider the effects of special (or general) relativity, but at the kinetic level such velocity corrections may play a fundamental role, even in the basic existence, uniqueness, and regularity properties of solutions.  Hence, one of the primary goals of the current work is to understand how this difference in formulation affects such properties, and yield a partial answer to the question, ``Are relativistic velocity corrections really necessary to ensure classical well-posedness?''.  It should be noted here that, whereas (RVM) is invariant under Lorentzian transformations, (\ref{VM}) lacks invariance properties as it combines a Galilean-invariant equation for the particle distribution with a Lorentz-invariant field equation.  To date, though, we are unaware of any argument which truly utilizes the invariance properties of (RVM) or (\ref{VM}) in order to arrive at an existence, uniqueness, or regularity theorem.

Often a remedy to the lack of progress on such a problem is to reduce the dimensionality of the system.  Unfortunately, posing the problem in one-dimension (i.e., $x,v \in \bfR$) eliminates the relevance of the magnetic field as the Maxwell system decouples, yielding the one-dimensional Vlasov-Poisson system:
\begin{equation} \tag{VP} \label{VP} \left \{ \begin{gathered}
 \partial_t f + v\partial_x f + E \partial_v f = 0\\
\partial_x E = \int f dv. \end{gathered} \right.
\end{equation}
The lowest-dimensional reduction which includes magnetic effects is the so-called ``one-and-one-half-dimensional'' system which is constructed by taking $x \in \bfR$ but $v \in \bfR^2$:
\begin{equation} \tag{1.5D VM} \label{1.5DVM} \left \{ \begin{gathered}
 \partial_t f + v\partial_x f + (E_1 + v_2 B) \partial_{v_1} f + (E_2 - v_1 B) \partial_{v_2} f = 0\\
 \partial_x E_1 = \int f dv, \qquad  \partial_t E_1 = - \int v_1 f dv \\
\partial_t (E_2 + B) + \partial_x(E_2 + B)  = - \int v_2 f dv \\
\partial_t (E_2 - B) - \partial_x(E_2 - B)  = - \int v_2 f dv. \end{gathered} \right.
\end{equation}
Surprisingly, the question of classical regularity remains open even in this simplified case.  The noticeable difference between (\ref{1.5DVM}) and (\ref{VP}) is the introduction of electric and magnetic fields $E_2$ and $B$ that are solutions to transport equations.  Thus, in order to study the existence and regularity questions, but keep the problem posed in a one-dimensional setting, we consider the following nonlinear system of PDE which couples the Vlasov equation to a transport field equation:
\begin{equation} \label{SVM} \left \{ \begin{gathered}
 \partial_t f + v\partial_x f + B \partial_v f = 0\\
\partial_t B + \partial_x B = \int f dv. \end{gathered} \right.
\end{equation}
The system (\ref{SVM}) is supplemented by given initial data
\begin{equation}
\label{IC}
f(0,x,v) = f_0(x,v),\quad  B(0,x) = B_0(x).
\end{equation}
Since the field equation in (\ref{SVM}) is hyperbolic, we denote it with a magnetic field variable $B$, as opposed to the electric field $E$ of (\ref{VP}) which satisfies an elliptic equation.  Notice that these equations retain the main difficulty of (\ref{VM}), namely the interaction between characteristic particle velocities $v$ and constant field velocities $c = 1$.  Hence, we hope to analyze (\ref{SVM}) and develop estimates or methods which can be generalized to deal with (\ref{1.5DVM}). 
To our knowledge, this is the first analytic study of these kinetic equations, formed from a system of conservation laws coupled by a non-local field dependence on the particle densities.  As such, the properties of solutions to (\ref{SVM}) may also be of interest to mathematicians studying scalar, hyperbolic conservation laws in a two-dimensional phase space.  
A related model, similar to (\ref{SVM}), was previously studied \cite{PG} in an attempt to understand the possible singularities generated by the intersection of Vlasov and transport characteristics.  However, only minor results were derived in this work, all of which concerned a reduced system of ODEs rather than PDEs.  Here, we present results for the full system and our results also generalize to their original system of PDEs.
We also mention the work \cite{BP} as it contains a discussion of computational methods for (\ref{SVM}).

This paper proceeds as follows.  In the next section, we will derive \emph{a priori} estimates in order to simplify the proof of the local-in-time existence and uniqueness of classical solutions to (\ref{SVM}) with given smooth initial data (\ref{IC}), which follows in Section $3$.  In Section $4$, we present two results concerning global existence, namely that certain particle distributions for which particle velocities travel at light speed do remain smooth globally in time.  The unfortunate detail of Theorems \ref{T1} and \ref{T2}, however, is that they do not necessarily extend to arbitrary initial data.  Hence, it still remains an open problem to show that any solution launched from smooth initial data remains smooth for all time.  As an intermediate step, we prove in Section $5$ an additional regularity result for the associated field $B$ using a decomposition of derivatives similar to that of \cite{GlStr}.  More specifically, we show \emph{a priori} that for any $T > 0$, $B \in C^{0,1/2}([0,T] \times \bfR)$. Throughout the paper the value $C > 0$ will denote a generic constant that may change from line to line.  When necessary, we will specifically identify the quantities upon which $C$ may depend (e.g., $C_T$).  Since we are interested in classical solutions, we will also assume the initial data are smooth and bounded, i.e. $f_0 \in C_c^1(\bfR^2)$ and $B_0 \in C^1(\bfR)$, for the entirety of the paper.

\section{A priori estimates}

To begin, we will first prove a lemma that will allow us to represent and bound the particle density, its derivatives, and the associated field.
\renewcommand{\labelenumi}{(\alph{enumi}).}

\begin{lemma}[Estimates on $f$, $B$, $\partial f$, and velocity support]
\label{L1}
\ \vspace{0.1in}
\begin{enumerate}
\item Let $T > 0$ be given and $f$ be the solution of the Vlasov equation with given initial data \begin{eqnarray*} \partial_t f + v \partial_x f + B \partial_v f = 0, &\qquad& t \in (0,T), \ x,v \in \bfR \\ f(0,x,v) = f_0(x,v), &\qquad& x,v \in \bfR \end{eqnarray*} for some given $B \in C^1([0,T] \times \bfR)$.  Then, $f  \in C^1([0,T]; C^1_c(\bfR^2))$ and for any $t \in [0,T]$ we have the estimates
$$\begin{gathered} \Vert f(t) \Vert_\infty \leq \Vert f_0\Vert _\infty, \\
\Vert \partial_x f(t) \Vert_\infty \leq \Vert \partial_x f_0 \Vert_\infty + C \int_0^t \Vert \partial_x B(s) \Vert_\infty ( 1 + s \sup_{\tau \in [0,s]} \Vert \partial_x f(\tau) \Vert_\infty ) \ ds, \\
\Vert \partial_v f(t) \Vert_\infty \leq \Vert \partial_v f_0 \Vert_\infty + \int_0^t \Vert \partial_x f(s) \Vert_\infty \ ds  \end{gathered} $$ where $C$ depends only upon the initial data.
\vspace{0.1in}
\item Let $T > 0$ be given and $B$ be the solution to the field equation with a given initial condition, namely \begin{eqnarray*} \partial_t B + \partial_x B = \int f \ dv,& \qquad & t \in (0,T), \ x \in \bfR\\ B(0,x) = B_0(x), & \qquad & x \in \bfR \end{eqnarray*} for some given $f \in C^1([0,T]; C^1_c(\bfR^2))$ and define $$P(t) = \sup \{ \vert v \vert : f(s, x, v) \neq 0, \ s \in [0,t], x \in \bfR \}.$$  Then, $B \in C^1([0,T] \times \bfR)$ and for any $t \in [0,T]$, we have the estimate
$$ \Vert B(t) \Vert_\infty \leq C \left (1 + \int_0^t P(s) \ ds \right )$$ where $C$ depends only upon the initial data.
\end{enumerate}
\end{lemma}
\renewcommand{\labelenumi}{\arabic{enumi}}

\begin{proof}
To prove the first result, we begin by introducing characteristics for the Vlasov equation.   Define the curves $X(s,t,x,v)$ and $V(s,t,x,v)$ as solutions to the system of ODEs
\begin{equation} \label{char} \left\{
\begin{gathered}
\frac{\partial X}{\partial s}=V(s,t,x,v), \\
X(t,t,x,v)=x, \\
\frac{\partial V}{\partial s}=B(s,X(s,t,x,v)), \\
V(t,t,x,v)=v
\end{gathered} \right.
\end{equation}
Often, the $(t,x,v)$ dependence of these curves will be suppressed so, for example, $X(s,t,x,v)$ will be denoted by $X(s)$ for brevity.
Then, the Vlasov equation can be expressed as a derivative along the characteristic curves by
\begin{eqnarray*}
\frac{d}{ds}\biggl (f(s, X(s), V(s)) \biggr )
& = & \partial_t f(s, X(s), V(s)) + \frac{\partial X}{\partial s}  \partial_x f(s, X(s), V(s))\\ && \qquad + \frac{\partial V}{\partial s} \partial_v f(s, X(s), V(s))\\
&=& \partial_t f(s, X(s), V(s)) + V(s) \partial_x f(s, X(s), V(s))\\ && \qquad + B(s,X(s))\partial_v f(s, X(s), V(s)) \\
&=& 0.
\end{eqnarray*}
Thus, we find $$f(t, X(t,t,x,v), V(t,t,x,v)) = f(0, X(0,t,x,v), V(0,t,x,v))$$
and $$f(t,x,v) = f_0(X(0,t,x,v), V(0,t,x,v)) \leq \Vert f_0 \Vert_\infty .$$ Finally, taking the supremum of this equality over $x,v \in \bfR$ yields the first estimate in $(a)$. The remaining estimates involve derivatives, so differentiating the Vlasov equation in $v$ yields
$$ \biggl ( \partial_t + v \partial_x + B\partial_v \biggr ) \partial_v f = -\partial_x f$$ and upon integrating along characteristics we find
\begin{equation}
\label{dvfn}
\partial_v f(t,x,v) = (\partial_v f_0)(X(0), V(0)) - \int_0^t (\partial_x f)(s,X(s),V(s)) \ ds.
\end{equation}
Taking supremums in $x$ and $v$ gives the last estimate for $(a)$.  The second estimate is derived similarly.  Differentiating with respect to $x$ in the Vlasov equation, we find $$\biggl ( \partial_t + v \partial_x + B\partial_v \biggr ) \partial_x f = -\partial_x B \partial_v f$$ and integrating as before along characteristics yields
\begin{equation}
\label{dxfn}
\partial_x f(t,x,v) = (\partial_x f_0)(X(0), V(0)) - \int_0^t (\partial_x B \partial_v f)(s,X(s),V(s)) \ ds.
\end{equation}
Taking the supremum on the right side, inserting the estimate on $\Vert \partial_v f (s) \Vert_\infty$ above, and using the bounded data gives us \begin{eqnarray*}
\vert \partial_x f(t,x,v) \vert & \leq & \Vert \partial_x f_0 \Vert_\infty + \int_0^t \Vert \partial_x B(s) \Vert_\infty \Vert \partial_v f(s) \Vert_\infty \ ds\\
& \leq & \Vert \partial_x f_0 \Vert_\infty + \int_0^t \Vert \partial_x B(s) \Vert_\infty  \left ( C + \int_0^s \Vert \partial_x f(\tau) \Vert_\infty \ d\tau \right) \ ds\\
& \leq & \Vert \partial_x f_0 \Vert_\infty + C \int_0^t \Vert \partial_x B(s) \Vert_\infty  \left ( 1 + s \sup_{\tau \in [0,s]} \Vert \partial_x f(\tau) \Vert_\infty \right) \ ds.
\end{eqnarray*}  Finally, taking the supremum gives the estimate for $\Vert \partial_x f (t) \Vert_\infty$ in $(a)$.

For the claim in $(b)$, we may use the method of characteristics to solve for $B$ in terms of $f$.  First, we write the differential equation for $B$ as a derivative along the curves $(s, x - t + s)$ so that $$ \frac{d}{ds} \biggl ( B(s, x - t + s) \biggr ) = \partial_t B(s, x-t+s) + \partial_x B(s,x-t+s) = \int f(s, x - t + s, v) \ dv. $$ Hence, integrating along these curves, we arrive at
\begin{equation}
\label{Brep}
B(t,x) = B_0(x-t) + \int_0^t \int f(s, x - t + s, v) \ dv \ ds.
\end{equation}
Using the bounded data, compact velocity support, and uniform bound on the particle density yields
\begin{eqnarray*}
\vert B(t,x) \vert & \leq & C + \int_0^t \int_{\{\vert v \vert : f \neq 0 \} } f(s, x - t + s, v) \ dv \ ds\\
& \leq & C \biggl ( 1 + \int_0^t \Vert f(s) \Vert_\infty P(s) \ ds \biggr ) \\
& \leq & C \biggl ( 1 + \int_0^t P(s) \ ds \biggr ).
\end{eqnarray*}
Finally, taking the supremum in $x$ concludes the proof.
\end{proof}

\section{Local-in-time existence of classical solutions}
With these \emph{a priori} estimates and classical convergence theorems from analysis, we now have the necessary tools to prove the local-in-time existence theorem.
\begin{theorem}[Existence of classical solutions]
\label{TMain}
Let $f_0 \in C_c^1(\bfR^2)$ and $B_0 \in C^1(\bfR)$ be given.  Then, there exists $T > 0 $ and a unique classical solution
$$ f  \in C^1([0,T]; C^1_c(\bfR^2)), \quad B \in C^1([0,T] \times \bfR)  $$ to (\ref{SVM}) satisfying the initial conditions (\ref{IC}).  Moreover, if we denote the maximal lifespan of the solution by $T^*$ then for $T^* < \infty$ we must have $$\limsup_{t \to T^*} \biggl ( \Vert \partial_x f(t) \Vert_\infty + \Vert \partial_v f(t) \Vert_\infty \biggr ) = \infty.$$
\end{theorem}
\begin{proof}
The outline of our proof generally follows the structure of \cite{KRST}, \cite{VPSSA1DLocal}, and \cite{VPSSA3DGeneral}.
We begin with the existence argument, which utilizes the method of successive approximations.  Hence, we define an iterative sequence of solutions to linear PDEs and show that it must converge to a solution of the nonlinear system (\ref{SVM}).  We take $f_0 \in C_c^1(\bfR^2)$ and $B_0 \in C^1(\bfR)$ and define
$$\begin{gathered}
f^0(t,x,v) = f_0(x,v),\\
B^0(t,x) = B_0(x).
\end{gathered}$$
Additionally, for every $n \in \bfN$, define $f^n \in C^1([0,\infty); C^1_c(\bfR^2))$ and $B^n \in C^1([0,\infty) \times \bfR)$ by solving the linear initial-value problems
\begin{equation}
\label{fnIC}
\left\{
\begin{gathered}
\partial_t f^n + v\partial_x f^n + B^{n-1}\partial_v f^n = 0\\
f^n(0,x,v) = f_0(x,v),
\end{gathered}
\right.
\end{equation}
and
\begin{equation}
\label{BnIC}
\left\{
\begin{gathered}
\partial_t B^n + \partial_x B^n = \int f^n dv\\
B^n(0,x) = B_0(x)
\end{gathered}
\right.
\end{equation}
respectively.  Notice that if $f^n \rightarrow f$ and $B^n \rightarrow B$ in the appropriate sense as $n \rightarrow \infty$ then $f$ and $B$ will satisfy (\ref{SVM}).  Now as in Lemma \ref{L1}, we further define the sequence of velocity support functions for every $t \geq 0, n \in \bfN$, $$P^n(t) = \sup \{ \vert v \vert : f^n(s, x, v) \neq 0, \ s \in [0,t], x \in \bfR \}.$$

\subsection{Uniform boundedness}
For the first portion of the proof, we consider $T > 0$ given and estimate on the time interval $[0,T]$.  In order to uniformly bound the sequence $B^n$, we utilize the estimates on the velocity support of $f^n$.  First, by Lemma \ref{L1}, we have the bound
\begin{equation}
\label{BP}
\Vert B^n(t) \Vert_\infty \leq C \left (1 + \int_0^t P^n(s) \ ds \right )
\end{equation} for every $n \in \bfN$.
To bound the velocity support in terms of the field, we express the solution of the Vlasov equation in terms of the associated characteristics.  For every $n \in \bfN$ define the characteristic curves $X^n(s,t,x,v)$ and $V^n(s,t,x,v)$ by
\begin{equation} \label{charn} \left\{
\begin{gathered}
\frac{\partial X^n}{\partial s}=V^n(s,t,x,v), \\
X^n(t,t,x,v)=x, \\
\frac{\partial V^n}{\partial s}=B^{n-1}(s,X^n(s,t,x,v)), \\
V^n(t,t,x,v)=v.
\end{gathered} \right.
\end{equation}
As before, the $(t,x,v)$ dependence of these curves will be suppressed for brevity. Then, using the argument in the proof of the first result of Lemma \ref{L1}, we find $$f^n(t, X^n(t,t,x,v), V^n(t,t,x,v)) = f^n(0, X^n(0,t,x,v), V^n(0,t,x,v))$$
and $$f^n(t,x,v) = f_0(X^n(0,t,x,v), V^n(0,t,x,v))$$ for every $n \in \bfN$.  
Inverting the characteristics using the identities
$$\begin{gathered}
x = X^n \left (t,0,X^n(0,t,x,v), V^n(0,t,x,v) \right)\\
v = V^n \left (t,0,X^n(0,t,x,v), V^n(0,t,x,v) \right)
\end{gathered}$$
within this last equality, then utilizing the compact support of $f_0$ and the definition of $P^n$ it follows that $$\sup_{x,v} \vert V^n(t,0,x,v) \vert \leq P^n(t)$$ for every $t \geq 0$. From the velocity characteristic equation (\ref{charn}), we can integrate to find $$V^n(t,0,x,v) = v + \int_0^t B^{n-1}(\tau, X^n(\tau)) \ d\tau$$ and thus $$ \vert V^n(t) \vert \leq \vert v \vert + \int_0^t \Vert B^{n-1}(\tau) \Vert_\infty \ d\tau.$$ Taking the supremum over characteristics along which $f^n \neq 0$, we find $$ P^n(t) \leq P^n(0) + \int_0^t \Vert B^{n-1}(\tau) \Vert_\infty \ d\tau.$$ We can now use (\ref{BP}) to arrive at a recursive bound for $P^n$, namely $$P^n(t) \leq P^n(0) + C \int_0^t \left (1 + \int_0^\tau P^{n-1}(s) \ ds \right ) \ d\tau.$$ Since $f_0$ has compact support, we know $P^n(0)$ is finite and constant in $n$, thus for every $n \in \bfN$ and on every bounded time interval $[0,T]$, $$P^n(t) \leq C \left ( 1 + \int_0^t P^{n-1}(\tau) \ d\tau \right )$$ where $C$ depends upon $f_0$ and $T$.  Using this recursive relation, we immediately deduce $$P^n(t) \leq C\left ( 1 +  \frac{t^n}{n!} \right ) \leq Ce^t \leq C_T.$$  Thus, on any bounded time interval $[0,T]$ the function $P^n(t)$ is uniformly bounded and from (\ref{BP}) so is $\Vert B^n(t) \Vert_\infty$.\\

\subsection{Uniform boundedness of derivatives}
Now we focus on obtaining uniform bounds on derivatives, sketching the proof for $x$ derivatives, with $t$ and $v$ derivatives following similarly.  From the definition of the iterates we can differentiate the representation for $B^n$ (\ref{Brep}) after integrating along characteristics with speed one, so that
\begin{equation}
\label{dxBf}
\partial_x B^n(t,x) = B^\prime_0(x-t) + \int_0^t \int \partial_x f^n(\tau, x-t+\tau, v) \ dv \ d\tau.
\end{equation}
Using the bound on the velocity support above, we arrive at
\begin{equation}
\label{DB}
\Vert \partial_x B^n(t) \Vert_\infty  \leq C_T \left ( 1 + \int_0^t \Vert \partial_x f^n(\tau) \Vert_\infty \ d\tau \right )
\end{equation}
for every $n \in \bfN$.  Using Lemma \ref{L1}(a) we also have the estimate $$ \Vert \partial_x f^n(t) \Vert_\infty \leq \Vert \partial_x f_0 \Vert_\infty + C \int_0^t \Vert \partial_x B^{n-1}(s) \Vert_\infty \left ( 1+ s \sup_{\tau \in [0,s]} \Vert \partial_x f^n(\tau) \Vert_\infty \right ) \ ds.$$ We combine these inequalities to find
\begin{eqnarray*}
\Vert \partial_x f^n(t) \Vert_\infty & \leq & C_T\left [1 +  \int_0^t \left ( 1 + \int_0^s \Vert \partial_x f^{n-1}(\tau) \Vert_\infty \ d\tau \right ) \left ( 1+ s \sup_{\tau \in [0,s]} \Vert \partial_x f^n(\tau) \Vert_\infty \right ) \ ds \right ]\\
& \leq & C_T\left [1 +  \int_0^t \left ( 1+ s \sup_{\tau \in [0,s]} \Vert \partial_x f^{n-1}(\tau) \Vert_\infty \right )\left ( 1+ s \sup_{\tau \in [0,s]} \Vert \partial_x f^n(\tau) \Vert_\infty \right ) \ ds \right ].
\end{eqnarray*}
Now, let $$F^n(t) = \max_{1\leq k \leq n} \sup_{\tau \in [0,t]} \Vert \partial_x f^k(\tau) \Vert_\infty.$$  With this definition, the previous inequality yields
\begin{equation}
\label{Fn}
F^n(t) \leq C_T \left (1 + \int_0^t [1 + s F^n(s)]^2 \ ds \right ).
\end{equation}
Hence, by induction $F^n(t) \leq F(t)$ for every $n \in \bfN$, $t \in [0,T^*)$, where $F(t)$ is the maximal solution of the integral equation corresponding to (\ref{Fn}):
\begin{equation}
\label{F}
F(t) = C_{T^*} \left (1 + \int_0^t [1 + s F(s)]^2 \ ds \right ).
\end{equation}
This solution exists on some time interval $[0,T^*)$ with $T^* > 0$ determined by $f_0$ and  $B_0$.  This yields a uniform bound on $\Vert \partial_x f^n(t) \Vert_\infty$ on $[0,T]$ for every $n \in \bfN$ and $T < T^*$.  Additionally, $\Vert \partial_x B^n(t) \Vert_\infty$ is bounded on the same interval by (\ref{DB}). The argument can be repeated in the same manner to bound $\Vert \partial_t B(t) \Vert_\infty$, $\Vert \partial_t f(t) \Vert_\infty$, and $\Vert \partial_v f(t) \Vert_\infty$ on the same time interval.\\

\subsection{Uniform Cauchy property}
For the remainder of the proof, we will estimate on the bounded time interval $[0,T]$, where $T \in (0,T^*)$.  To show that the sequences $f^n$ and $B^n$ are Cauchy, we directly estimate differences of terms of the sequences. Let $n,m \in \mathbb{N}$ be given and define the functions $$f^{n,m}(t,x,v) = f^{n}(t,x,v) - f^{m}(t,x,v)$$ and $$B^{n,m}(t,x) = B^{n}(t,x) - B^{m}(t,x).$$ Since (\ref{fnIC}) holds for any $n \in \mathbb{N}$, we subtract the $f^{m}$ equation from that for $f^n$ to find
\begin{eqnarray*}
0 & = &  \partial_t f^{n,m} + v \partial_x f^{n,m} + B^{n-1} \partial_v f^{n} - B^{m-1} \partial_v f^{m}\\
& = &  \partial_t f^{n,m} + v \partial_x f^{n,m} + B^{n-1} \partial_v f^{n} - B^{n-1} \partial_v f^{m} + B^{n-1} \partial_v f^{m}- B^{m-1} \partial_v f^{m}\\
& = &  \partial_t f^{n,m} + v \partial_x f^{n,m} + B^{n-1} \partial_v f^{n,m} + B^{n-1,m-1} \partial_v f^{m}
\end{eqnarray*}
so that by rearranging terms this becomes $$ \partial_t f^{n,m} + v \partial_x f^{n,m} + B^{n-1} \partial_v f^{n,m} = - B^{n-1,m-1} \partial_v f^{m}.$$
As for the Vlasov equation, the left side of this equation can be expressed as a derivative along characteristic curves (\ref{charn}) as
$$\frac{d}{ds} f^{n,m} \bigl (s,X^{n-1}(s),V^{n-1}(s) \bigr )  = - \biggl ( B^{n-1, m-1}\partial_vf^{m} \biggr ) \bigl (s,X^{n-1}(s),V^{n-1}(s) \bigr ).$$ Now, integrating both sides with respect to $s$, we find
\begin{eqnarray*} f^{n,m}(t,x,v) & = & f^{n,m} \bigl (0, X^{n-1}(0), V^{n-1}(0) \bigr )\\
& & \qquad -  \int_0^t \bigl (B^{n-1,m-1}\partial_vf^{m} \bigr ) \bigl (s,X^{n-1}(s),V^{n-1}(s) \bigr ) ds
\end{eqnarray*}
and since both $f^n$ and $f^m$ satisfy the same initial condition (\ref{fnIC}), it follows that $f^{n,m}(0,x,v) = 0$.  Therefore, the equality simplifies to
$$f^{n,m}(t,x,v) = - \int_0^t \bigl (B^{n-1,m-1}\partial_vf^{m} \bigr ) \bigl (s,X^{n-1}(s),V^{n-1}(s) \bigr ) ds.$$ We know $\Vert \partial_v f^{m}(s) \Vert_\infty$ is uniformly bounded (from Section $3.2$) so we can bound the right side to find
\begin{equation}
\label{fBn}
\Vert f^{n,m}(t) \Vert_\infty \leq C \int_0^t \Vert B^{n-1,m-1}(s) \Vert_\infty \ ds.
\end{equation}
Now, since (\ref{BnIC}) must also hold for all $n \in \bfN$, we subtract the $B^{m}$ equation from that for $B^{n}$ and arrive at
$$\partial_t B^{n,m} + \partial_x B^{n,m} = \int f^{n,m} \ dv$$ which we can write as a derivative along curves with slope one as
$$\frac{d}{ds} B^{n,m}(s, x - t + s) = \int f^{n,m}(s, x - t + s, v) \ dv.$$  Integrating in $s$ and using the initial conditions (\ref{BnIC}) to conclude that $B^{n,m}(0,x) = 0$ for every $x \in \bfR$, this becomes
$$B^{n,m}(t,x) = \int_0^t \int f^{n,m}(s, x - t + s, v)\ dvds.$$  Since the velocity support of $f$, denoted by $P^n$, is uniformly bounded from Section $3.11$, we can bound $B^{n,m}$ by
\begin{equation}
\label{Bfn}
\Vert B^{n,m}(t) \Vert_\infty \leq C_T \Vert f^{n,m}(t) \Vert_\infty
\end{equation}
Finally, combining (\ref{fBn}) and (\ref{Bfn}) yields
$$ || B^{n,m}(t)  ||_\infty \leq C \int_0^t ||B^{n-1,m-1}(s)||_\infty ds$$ for any $n,m \in \mathbb{N}$. Now consider $m=n-1$ so that the previous equation becomes
$$ || B^{n,n-1}(t)  ||_\infty \leq C \int_0^t ||B^{n-1,n-2}(s)||_\infty ds$$  for any $n \in \bfN$ with $n \geq 2$. Using this recursive relation for successive differences, we deduce
$$ || B^{n,n-1}(t) ||_\infty \leq C\Vert B^{1,0}(t) \Vert_\infty \frac{t^{n-1}}{(n-1)!}.$$ Since $\Vert B^1(t) \Vert_\infty$ and $\Vert B^0(t) \Vert_\infty$ are bounded, we have $$\Vert B^{n,m}(t) \Vert_\infty \leq \sum_{k=m+1}^n \Vert B^{k,k-1}(t) \Vert_\infty \leq C \sum_{k=m+1}^n \frac{t^{k-1}}{(k-1)!} \rightarrow 0$$ as $n,m \rightarrow \infty$ for every $t \in [0,T]$. Therefore, $B^n(t,x)$ is uniformly Cauchy for all $t \in [0,T]$ and $x \in \bfR$. Similarly, using (\ref{fBn}) we see that $ \Vert f^{n,m}(t) \Vert_\infty \rightarrow 0$ as $n \rightarrow \infty$ for every $t \in [0,T]$ and it follows that $f^n(t,x,v)$ is uniformly Cauchy for every $t \in [0,T]$ and $x,v \in \bfR$.\\

To conclude this section, we use the newly discovered boundedness of derivatives and Cauchy property of the field to show that the characteristics given by (\ref{charn}) are uniformly Cauchy.  First, we let $m, n \in \bfN$ be given and define $$X^{n,m}(s) = X^n(s) - X^m(s)$$ with the analogous definition for $V^{n,m}(s).$ Upon integrating the ODEs of (\ref{charn}) and subtracting the equations for $X^m$ from those of $X^n$, we find
\begin{equation}
\label{Xnm}
X^{n,m}(s) = - \int_s^t V^{n,m}(\tau) \ d \tau$$ and thus $$\Vert X^{n,m}(s) \Vert_\infty \leq \int_s^t \Vert V^{n,m}(\tau) \Vert_\infty d\tau.
\end{equation}
Doing the same for the $V^n(s)$ equations, we simply use the Mean Value Theorem to find
\begin{eqnarray*}
\vert V^{n,m}(s) \vert & = & \left \vert \int_s^t  \biggl ( B^n(\tau, X^n(\tau) ) - B^m(\tau,X^m(\tau)) \biggr ) \ d\tau \right \vert\\
& = &   \left \vert \int_s^t \biggl ( B^n(\tau, X^n(\tau) ) - B^n(\tau, X^m(\tau) ) + B^n(\tau, X^m(\tau) ) - B^m(\tau,X^m(\tau)) \biggr ) \ d\tau \right \vert\\
& \leq & \int_s^t \biggl ( \Vert \partial_x B^n (\tau) \Vert_\infty \vert X^n(\tau) - X^m(\tau) \vert + \Vert B^{n,m} (\tau) \Vert_\infty \biggr ) \ d\tau.
\end{eqnarray*}
Since field derivatives are bounded from Section $3.2$, this implies
\begin{equation}
\label{Vnm}
\Vert V^{n,m}(s) \Vert_\infty \leq C_T \int_s^t \biggl (\Vert X^{n,m}(\tau)\Vert_\infty + \Vert B^{n,m} (\tau) \Vert_\infty \biggr ) \ d\tau.
\end{equation}
Finally, we let $$Z^{n,m}(s) = \Vert X^{n,m}(s) \Vert_\infty + \Vert V^{n,m}(s) \Vert_\infty$$ and add (\ref{Xnm}) and (\ref{Vnm}) together to find $$Z^{n,m}(s) \leq C\int_s^t Z^{n,m}(\tau) \ d\tau + C\int_s^t \Vert B^{n,m} (\tau) \Vert_\infty \ d\tau.$$ We then use Gronwall's Inequality (cf. \cite{Evans}) to deduce the bound $$Z^{n,m}(s) \leq C\int_s^t \Vert B^{n,m} (\tau) \Vert_\infty \ d\tau.$$ Since we know $\Vert B^{n,m}(\tau) \Vert_\infty \to 0$ uniformly for $\tau \in [0,T]$, this implies that $Z^{n,m}(s)$ does so as well, and finally that $X^n$ and $V^n$ are uniformly Cauchy.\\

\subsection{Uniform Cauchy property of derivatives}
In order to prove that the resulting limits of $f^n$ and $B^n$ are differentiable, we will show that the sequence of derivatives $\partial_{(t,x,v)} f^n$ and $\partial_{(t,x)} B^n$ are Cauchy as well.
We begin by bounding $\Vert \partial_x B^{n,m}(t) \Vert_\infty$.  Using the representation from (\ref{dxBf}) and subtracting the equation for $\partial_x B^{m}$ from that for $\partial_x B^{n}$ we obtain
$$ \partial_x B^{n,m}(t,x) = \int_0^t \int \partial_x f^{n,m}(\tau,x - t + \tau, v) \ dv \ d\tau.$$
Using the boundedness of the velocity support from Section $3.1$, this relationship implies
\begin{equation}
\label{dxBnm}
\Vert \partial_x B^{n,m}(t) \Vert_\infty  \leq C_T \int_0^t \Vert \partial_x f^{n,m}(s) \Vert_\infty \ ds.
\end{equation}
Hence, we must estimate derivatives of $f^{n,m}$ also.  We first use the representation for $\partial_v f^n$ from (\ref{dvfn}) and subtract this equation for $\partial_v f^{m}$ from that of $\partial_v f^{n}$ to find
\begin{eqnarray*}
\vert \partial_v f^{n,m}(t,x,v) \vert & \leq & \vert \partial_v f_0(X^n(0),V^n(0)) -  \partial_v f_0(X^m(0),V^m(0)) \vert \\
& & + \int_0^t \vert \partial_x f^n(s,X^n(s),V^n(s)) - \partial_x f^m(s,X^m(s),V^m(s)) \vert\\
& \leq & \vert \partial_v f_0(X^n(0),V^n(0)) -  \partial_v f_0(X^m(0),V^m(0)) \vert \\
& & + \int_0^t \vert \partial_x f^n(s,X^n(s),V^n(s)) - \partial_x f^n(s,X^m(s),V^m(s)) \vert\\
& & + \int_0^t \Vert \partial_x f^{n.m}(s) \Vert_\infty \ ds
\end{eqnarray*}
Here, the first and second terms tend to zero uniformly as $n,m \to \infty$ using the Cauchy property of characteristics and the uniform boundedness of derivatives of the iterates. Thus, taking supremums, we find
\begin{equation}
\label{dvfnm}
\Vert \partial_v f^{n,m}(t) \Vert_\infty \leq \alpha_{n,m} + \int_0^t \Vert \partial_x f^{n,m}(s) \Vert_\infty \ ds
\end{equation}
where $\alpha_{n,m} \to 0$ uniformly on $[0,T]$ as $n,m \to \infty$.

Similarly, we use the representation for $\partial_x f^n$ from (\ref{dxfn}) and take the difference of this equation for $\partial_x f^{n}$ and $\partial_x f^{m}$ to find
\begin{eqnarray*}
\vert \partial_x f^{n,m}(t,x,v) \vert & = & \left \vert \partial_x f_0 (X^n(0),V^n(0)) - \partial_x f_0 ( X^m(0), V^m(0)) \right. \\
& & + \int_0^t (\partial_x B^{n-1} \partial_v f^n)(s,X^n(s),V^n(s)) \ ds \\
& & \left. - \int_0^t (\partial_x B^{m-1} \partial_v f^m)(s,X^m(s),V^m(s)) \ ds \right \vert\\
& \leq & \vert \partial_x f_0 (X^n(0),V^n(0)) - \partial_x f_0 ( X^m(0), V^m(0)) \vert\\
& &  + \int_0^t \biggl [ \vert (\partial_x B^{n-1} \partial_v f^n)(s,X^n(s),V^n(s)) - (\partial_x B^{m-1} \partial_v f^n)(s,X^n(s),V^n(s)) \vert \\
& & + \vert (\partial_x B^{m-1} \partial_v f^n)(s,X^n(s),V^n(s)) - (\partial_x B^{m-1} \partial_v f^n)(s,X^m(s),V^m(s))\vert \\
& & + \vert (\partial_x B^{m-1} \partial_v f^n)(s,X^m(s),V^m(s)) - (\partial_x B^{m-1} \partial_v f^m)(s,X^m(s),V^m(s))\vert \biggr ]  ds\\
& \leq & \vert \partial_x f_0 (X^n(0),V^n(0)) - \partial_x f_0 ( X^m(0), V^m(0)) \vert\\
& & + \int_0^t \biggl [ \biggl ( \vert \partial_x B^{n-1} - \partial_x B^{m-1} \vert \cdot \vert\partial_v f^n \vert \biggr )(s,X^n(s),V^n(s)) \\
& & + \vert (\partial_x B^{m-1} \partial_v f^n)(s,X^n(s),V^n(s)) - (\partial_x B^{m-1} \partial_v f^n)(s,X^m(s),V^m(s))\vert \\
& & + \biggl (\vert \partial_x B^{m-1} \vert \cdot \vert \partial_v f^n - \partial_v f^m \vert \biggr )(s,X^n(s),V^n(s)) \biggr ] \ ds\\
& \leq & \vert \partial_x f_0 (X^n(0),V^n(0)) - \partial_x f_0 ( X^m(0), V^m(0)) \vert\\
& & + \int_0^t \vert (\partial_x B^{m-1} \partial_v f^n)(s,X^n(s),V^n(s)) - (\partial_x B^{m-1} \partial_v f^n)(s,X^m(s),V^m(s))\vert \ ds \\
& & + \int_0^t  \left ( \Vert \partial_x B^{n-1,m-1}(s) \Vert_\infty \Vert \partial_v f^{n}(s) \Vert_\infty +\Vert \partial_x B^{m-1}(s) \Vert_\infty \Vert \partial_v f^{n,m}(s) \Vert_\infty  \right ) \ ds.\\
\end{eqnarray*}
From Section $3.2$ we have uniform bounds on $ \Vert \partial_v f^{n}(s) \Vert_\infty $ and $ \Vert \partial_x B^{m-1}(s) \Vert_\infty$.  Thus, the first and second terms tend to zero uniformly as $n,m \to \infty$ and the last line simplifies.  With this, we have
$$\Vert \partial_x f^{n,m}(t) \Vert_\infty  \leq \beta_{n,m} + C_T \int_0^t \left [ \Vert \partial_x B^{n-1,m-1}(s) \Vert_\infty +  \Vert \partial_v f^{n,m}(s) \Vert_\infty \right ] \ ds$$
where $\beta_{n,m} \to 0$ uniformly on $[0,T]$ as $n,m \to \infty$.
Using (\ref{dvfnm}), this inequality becomes
$$\Vert \partial_x f^{n,m}(t) \Vert_\infty  \leq C_T\left (\alpha_{n,m} + \beta_{n,m} +  \int_0^t \left [ \Vert \partial_x B^{n-1,m-1}(s) \Vert_\infty +  \int_0^s \Vert \partial_x f^{n,m}(\tau) \Vert_\infty d\tau \right ] ds \right )$$
which simplifies to (with $\gamma_{n,m} = \alpha_{n,m} + \beta_{n,m}$)
$$\Vert \partial_x f^{n,m}(t) \Vert_\infty  \leq C_T \left (\gamma_{n,m} + \int_0^t \Vert \partial_x B^{n-1,m-1}(s) \Vert_\infty \ ds +  t\int_0^t \Vert \partial_x f^{n,m}(\tau) \Vert_\infty  \ d\tau \right ). $$
Using Gronwall's Inequality we find
\begin{equation} \label{dxfnm}
\Vert \partial_x f^{n,m}(s) \Vert_\infty \leq  C_T \left ( \gamma_{n,m} +   \Vert \partial_x B^{n-1,m-1}(s) \Vert_\infty \right ) .
\end{equation}
Combining (\ref{dxBnm}) and (\ref{dxfnm}) yields
\begin{equation}
\label{Bnm}
\Vert \partial_x B^{n,m}(t) \Vert_\infty  \leq C\left ( \gamma_{n,m} +  \int_0^t \Vert \partial_x B^{n-1,m-1}(s) \Vert_\infty \ ds \right )
\end{equation}
for any $n,m \in \mathbb{N}$. We can iterate (\ref{Bnm}) to arrive at
$$\Vert \partial_x B^{n,m}(t) \Vert_\infty  \leq \gamma_{n,m} \sum_{j=0}^{k-1} \frac{(Ct)^j}{j!} +  \frac{C^{k-1}}{(k-1)!}\int_0^t (t -s)^{k-1} \Vert \partial_x B^{n-k,m-k}(s) \Vert_\infty \ ds $$  for any $k \in \bfN$ with $1 \leq k \leq \min\{n,m\}$.
Thus we have $$\Vert \partial_x B^{n,m}(t) \Vert_\infty  \leq C_T \gamma_{n,m} + 2\frac{(CT)^k}{k!} \sup_{s \in [0,T]} \Vert \partial_x B_0(s) \Vert_\infty $$ for $k = \min\{n,m\} > 0$. Therefore, $\partial_x B^n(t,x)$ is uniformly Cauchy for all $t \in [0,T]$ and $x \in \bfR$. Similarly, using (\ref{dvfnm}) and (\ref{dxfnm}) we see that $ \Vert \partial_{(x,v)} f^{n,m}(t) \Vert_\infty \rightarrow 0$ as $n,m \rightarrow \infty$ for every $t \in [0,T]$ and it follows that $\partial_{(x,v)} f^n(t,x,v)$ is uniformly Cauchy for every $t \in [0,T], x,v \in \bfR$.  The same argument can then be used to show that $\partial_t f^n(t,x,v)$ and $\partial_t B(t,x)$ are also uniformly Cauchy.\\

\subsection{Properties of limiting functions}
Assembling the previous steps, we may prove that our sequences and their derivatives converge to solutions of (\ref{SVM}).  Let $T^*$ again denote the maximal existence time of the solution to (\ref{F}).  Using the Cauchy property of $f^n$, $X^n$, $V^n$, $B^n$, and their derivatives, we may conclude that each sequence of functions converges uniformly on the time interval $[0,T]$ for any $T < T^*$ and uniformly for $x,v \in \bfR, n \in \bfN$.  Since the space of continuous functions is complete with respect to the norm of uniform convergence, we may conclude that these sequences converge to continuous functions.  Therefore, let us define $f \in C( [0,T] \times \bfR^2 )$ by $$f(t,x,v) = \lim_{n \to \infty} f^n(t,x,v) = \lim_{n \to \infty} f_0 \bigl ( X^n(0,t,x,v), V^n(0,t,x,v) \bigr ).$$  Then, we similarly define the field $$B(t,x) = \lim_{n \to \infty} B^n(t,x)= \lim_{n \to \infty} \biggl (B_0(x-t) + \int_0^t \int f^n(s, x - t + s, v) \ dv \ ds \biggr ).$$
Thus, using the uniform convergence of $f^n(t,x,v)$, we can pass the limit inside these integrals to find for every $t \in [0,T]$, $x \in \bfR$
\begin{equation}
\label{Bfinal}
B(t,x) = B_0(x-t)  + \int_0^t \int f(s,x-t+s,v) \ dv \ ds.
\end{equation}
Further, we define $$X = \lim_{n \to \infty} X^n, \qquad V = \lim_{n \to \infty} V^n.$$ It follows from (\ref{charn}) and the uniform field bound that
\begin{equation} \label{charL} \left\{
\begin{gathered}
\frac{\partial X}{\partial s}=V(s,t,x,v), \\
\frac{\partial V}{\partial s}=B(s,X(s,t,x,v)), \\
X(t,t,x,v)=x, \\
V(t,t,x,v)=v,
\end{gathered} \right.
\end{equation}
and by the continuity of $f_0$, we see that $$f(t,x,v) = f_0(X(0,t,x,v), V(0,t,x,v)),$$ whence for every $s \in [0,t]$,
\begin{equation}
\label{fs}
f(t,x,v) = f(s,X(s,t,x,v), V(s,t,x,v)).
\end{equation}
Since the approximating sequence of derivatives (e.g., $\partial_x f^n$) of these functions converge uniformly, this implies that $f$ and $B$ are $C^1$ and their derivatives are necessarily the limits of the respective sequences, meaning $$\begin{gathered}\partial_x f = \lim_{n \to \infty} \partial_x f^n, \quad \partial_v f = \lim_{n \to \infty} \partial_v f^n, \quad \partial_t f = \lim_{n \to \infty} \partial_t f^n,\\ \partial_x B = \lim_{n \to \infty} \partial_x B^n,\qquad \partial_t B = \lim_{n \to \infty} \partial_t B^n. \end{gathered}$$
Furthermore, the uniform bound on $P^n(t)$ implies the compact $x$ and $v$ support of $f(t,x,v)$ for every $t \in [0,T]$.  Thus, for every $T < T^*$, $f \in C^1([0,T]; C_c^1(\bfR^2))$. Using (\ref{Bfinal}) and taking derivatives, we see that the field equation for $B$ of (\ref{SVM}) holds.  Upon taking the derivative with respect to $s$ in (\ref{fs}) and using (\ref{charL}), we see that the Vlasov equation of (\ref{SVM}) holds.  Additionally, the solutions (\ref{fs}) and (\ref{Bfinal}) satisfy the initial conditions (\ref{IC}).  Therefore, the continuously differentiable functions $f$ and $B$ satisfy (\ref{SVM}) with (\ref{IC}).  Hence, we have shown the existence of such a solution $(f,B)$. Notice that this argument can be continued to a time interval of arbitrary size so long as $\Vert \partial_x f(t) \Vert_\infty$ and $\Vert \partial_v f(t) \Vert_\infty$ remain bounded. \\

\subsection{Uniqueness of solutions}
Finally, we turn to uniqueness.  Let us first suppose that the functions $(f^{(1)}, B^{(1)})$ and $(f^{(2)}, B^{(2)})$ are two solutions to the system (\ref{SVM}) on some time interval $[0,T]$ which satisfy (\ref{IC}).  Also, for every $t \in [0,T]$ and $x,v \in \mathbb{R}$ define the difference of these solutions $$\begin{gathered} f(t,x,v) = f^{(1)}(t,x,v)-f^{(2)}(t,x,v)\\ B(t,x) = B^{(1)}(t,x)-B^{(2)}(t,x). \end{gathered}$$

Then, we subtract the first equation of (\ref{SVM}) for $f^{(2)}$ from that for $f^{(1)}$ to find
\begin{eqnarray*}
0 & = &  \partial_t f + v \partial_x f + B^{(1)} \partial_v f^{(1)} - B^{(2)} \partial_v f^{(2)}\\
& = & \partial_t f + v \partial_x f + B^{(1)} \partial_v f^{(1)} - B^{(1)} \partial_v f^{(2)} + B^{(1)} \partial_v f^{(2)} - B^{(2)} \partial_v f^{(2)}\\
& = &  \partial_t f + v \partial_x f + B^{(1)} \partial_v f - B \partial_v f^{(2)}
\end{eqnarray*}
so that by rearranging terms this becomes $$ \partial_t f+ v\partial_xf + B^{(1)}\partial_vf =  B\partial_vf^{(2)}.$$
The left side of this equation can be expressed as a derivative along characteristic curves as
$$\frac{d}{ds} f \bigl (s,X^{(1)}(s),V^{(1)}(s) \bigr )  = \biggl ( B\partial_vf^{(2)} \biggr ) \bigl (s,X^{(1)}(s),V^{(1)}(s) \bigr )$$ where the curves $X^{(1)}(s)$ and $V^{(1)}(s)$ are defined by the, now well-known, system of characteristic ordinary differential equations
\begin{equation} \label{charU} \left\{
\begin{gathered}
\frac{\partial X^{(1)}}{\partial s}=V^{(1)}(s,t,x,v), \\
\frac{\partial V^{(1)}}{\partial s}=B^{(1)}(s,X^{(1)}(s,t,x,v)), \\
X^{(1)}(t,t,x,v)=x, \\
V^{(1)}(t,t,x,v)=v.
\end{gathered} \right.
\end{equation}
Here, we have abbreviated $X^{(1)}(s,t,x,v)$ by $X^{(1)}(s)$ and similarly for $V^{(1)}(s)$.  Now, integrating both sides of the above equation with respect to $s$, we find
$$f(t,x,v) - f \bigl (0, X^{(1)}(0), V^{(1)}(0) \bigr ) = \int_0^t \bigl (B\partial_vf^{(2)} \bigr ) \bigl (s, X^{(1)}(s), V^{(1)}(s) \bigr ) ds$$ and since both solutions satisfy the same initial condition (\ref{IC}), as before this implies $f(0,x,v) \equiv 0$.  Therefore, the equality simplifies to
$$f(t,x,v) =\int_0^t \bigl (B\partial_vf^{(2)} \bigr ) \bigl (s, X^{(1)}(s), V^{(1)}(s) \bigr ) ds.$$ Since $f^{(2)}$ is a solution, we know $\Vert \partial_v f^{(2)}(s) \Vert_\infty$ is bounded for $s \in [0,t]$ and we can bound the right side to find
\begin{equation}
\label{fB}
\Vert f(t) \Vert_\infty \leq C_T \int_0^t \Vert B(s) \Vert_\infty \ ds.
\end{equation}
Now, by subtracting the $B^{(2)}$ equation from the $B^{(1)}$ equation, we arrive at $$\partial_t B + \partial_x B = \int f \ dv$$ which we can write as a derivative along curves with slope one as $$\frac{d}{ds} B(s, x - t + s) = \int f(s, x - t + s, v) \ dv.$$  Integrating in $s$ and using (\ref{IC}) to conclude that $B(0,x) \equiv 0$, this becomes
$$B(t,x) = \int_0^t \int f(s, x - t + s, v)\ dvds.$$  Since both $f^{(1)}$ and $f^{(2)}$ are solutions, the velocity support of $f$ is controlled and we can bound $B$ by
\begin{equation}
\label{Bf}
\Vert B(t) \Vert_\infty \leq C_T \sup_{s \in [0,t]} \Vert f(s) \Vert_\infty
\end{equation}
Finally, combining (\ref{fB}) and (\ref{Bf}) yields 
$$ || B(t)  ||_\infty \leq C_T \int_0^t \sup_{\tau \in [0,s]} ||B(\tau)||_\infty ds$$ 
and
$$   \sup_{s \in [0,t]} || B(s)  ||_\infty \leq C_T \int_0^t \sup_{\tau \in [0,s]} ||B(\tau)||_\infty ds$$ 
Again using Gronwall's Inequality, we deduce $ || B(t) ||_\infty \leq 0$ for all $t \in [0,T]$ which implies that $B(t,x) = 0$ for every $t \in [0,T]$, $x \in \bfR$. Similarly, using (\ref{fB}) we see that $f(t,x,v) = 0$ for every $t \in [0,T]$, $x,v \in \bfR$.  Finally, this implies that $f^{(1)} \equiv f^{(2)}$ and $B^{(1)} \equiv B^{(2)}$, and hence there can be at most one such solution.
\end{proof}

\section{Global Existence}
Now that we know smooth solutions exist on some time interval, the next logical question is whether this can be extended for all times.  Unfortunately, a complete answer remains unknown.  The fundamental issue is that the Vlasov characteristics in the density equation propagate at an uncontrollable speed $v$ and hence, are able to intersect the field characteristics which propagate with speed $1$.  Though we cannot current prove that all initial data launch a global-in-time solution, we can provide an answer for certain classes of initial data.  In what follows we will use the unidirectional nature of the transport operator in (\ref{SVM}) to answer the question of global existence in the affirmative for a class of initial data $(f_0,B_0)$.  Then, we shall utilize a new scaling invariance of the system to apply the global existence result to additional solutions.  We begin with the former result:

\begin{theorem}
\label{T1}
Let $f_0 \in C^1_c(\bfR^2)$ and $B_0 \in C^1(\bfR)$ be nonnegative with $\supp(f_0(x,\cdot)) \subset (1,\infty)$ for all $x \in \bfR$.  Then, for all $T > 0$ there exist $f \in C^1([0,T] \times \bfR^2)$ and $B \in C^1([0,T]\times\bfR)$ satisfying (\ref{SVM}) with (\ref{IC}) and $f(t,\cdot,\cdot)$ compactly supported for every $t \in [0,T]$.
\end{theorem}
\begin{proof}
Let $f$ and $B$ be the local solution guaranteed by Theorem \ref{TMain} and $T > 0$ be given. First, we represent $B$ and use the sign of the data to find
\begin{eqnarray*}
B(t,x) & = & B_0(x-t) + \int_0^t \int f(s,x+t-s,v) \ dv \ ds \\
& = & B_0(x-t) + \int_0^t \int f_0(X(s,t,x+t-s,v), V(s,t,x+t-s,v)) \ dv \ ds\\
& \geq & 0
\end{eqnarray*}
Thus, if we now let $(X(t),V(t))$ be characteristics along which $f$ is nonzero.  Then, a lower bound on velocity characteristics follows
$$V(t)  = V(0) + \int_0^t B(s,X(s)) \ ds \geq V(0).$$
Since the velocity support of $f_0$ is strictly bounded below by $v=1$, it follows that the velocity support of $f$ satisfies this same property.  

With this, we may utilize the field representation of \cite{GlStr} to control derivatives of the density and field (as in the proof of (\ref{T3})). 
Thus, we write the field derivative in terms of the derivative of the density, as in (\ref{dxBf}), so that
\begin{equation}
\label{DBDf}
\partial_x B(t,x) = B^\prime_0(x-t) + \int_0^t \int \partial_x f(\tau, x-t+\tau, v) \ dv \ d\tau.
\end{equation}
Now, we would like to eliminate the $x$-derivative of the density in this equation, so similar to \cite{GlaSch90} we transform $\partial_x$ into derivatives along the characteristic curves of the system.  Define the operators $$\begin{gathered} \mcS = \partial_t + v \partial_x\\ \mcT = \partial_t + \partial_x \end{gathered}.$$ Then, for $v \neq 1$ we may write the inverse transformation $$\partial_x = \frac{1}{1-v} \biggl ( \mcT - \mcS \biggr ).$$  For $t \in [0,T]$, let $$P(t) = \sup \{ \vert v \vert : \exists x \in \bfR \mathrm{ \ with \ } f(t,x,v) \neq 0 \}$$ and recall $P(t) \leq C_T$ for all $t \in [0,T]$.
Then, differentiation of the representation for the field (\ref{Bfinal}) yields
\begin{eqnarray*}
\partial_x B(t,x) & = & B^\prime_0(x-t) + \int_0^t \int \frac{1}{ 1- v} \left [ \mcT f(\tau, x-t+\tau, v) -  \mcS f(\tau, x-t+\tau, v) \right ] \ dv \ d\tau\\
& = & B^\prime_0(x-t) + \int_0^t \int \frac{1}{1-v} \frac{d}{d\tau} \left [f(\tau, x-t+\tau, v) \right ] \ dv \ d\tau\\
& \ & \  +  \int_0^t \int \frac{1}{1-v} \partial_v \left [B(\tau, x-t+\tau) f(\tau, x-t+\tau, v) \right ] \ dv \ d\tau
\end{eqnarray*}
Here we have used the Vlasov equation of (\ref{SVM}) to write the integrand as a pure $v$-derivative so that $$ \mcS f(\tau, y, v) =(\partial_t f + v \partial_x  f)(\tau,y,v) = - \partial_v ( Bf)(\tau,y,v).$$
Since the velocity support of $f$ is bounded away from $v=1$, we see that the integrands are non-singular and we integrate by parts to find
\begin{eqnarray*}
\partial_x B(t,x) & = &  B^\prime_0(x-t) + \int \frac{1}{1-v} [ f(t,x,v) - f_0(x-t,v)] \ dv\\
& \ &  + \int_0^t  \frac{1}{1-v} (Bf)(\tau, x-t+\tau, v) \biggr\vert_{v = P(t)} d\tau\\
& \ &  - \int_0^t  \int \frac{1}{(1-v)^2} (Bf)(\tau, x-t+\tau, v) \ dv \ d\tau.
\end{eqnarray*}
Using the previously-derived $L^\infty$ bounds on $f$, $P$, and $B$, which follow from the iterates, we find
$$\Vert \partial_x B(t) \Vert_\infty \leq C_T$$
Hence, using Lemma \ref{L1}(a) we find
$$ \Vert \partial_x f(t) \Vert_\infty \leq C_T \left (1 + \int_0^t \sup_{\tau \in [0,s]} \Vert \partial_x f(\tau) \Vert_\infty \ ds. \right ) $$
Taking the supremum in $t$ and using Gronwall's Inequailty implies
$$ \Vert \partial_x f(t) \Vert_\infty \leq C_T$$ and using \ref{L1}(a) again we find
$$ \Vert \partial_v f(t) \Vert_\infty \leq C_T$$ for any $T > 0$ and $t \in [0,T]$.  Thus, the local-in-time solution can be extended to arbitrarily large time.
\end{proof}

Next, we utilize a new invariance of (\ref{SVM}) to extend this result to additional solutions.
\begin{lemma}
\label{L2}
Let $T > 0$ be given. The functions $f \in C^1([0,T]; C_c^1(\bfR^2))$ and $B \in C^1([0,T]\times \bfR)$ solve (\ref{SVM}) for $t \in [0,T], x, v \in \bfR$ with initial data (\ref{IC}) if and only if the functions
\begin{equation}
\label{fu}
f^u(t,x,v) := f(t,(u+1)x - ut, (u+1)v -u)
\end{equation}
and
\begin{equation}
\label{Bu}
B^u(t,x) := (u+1)^{-1} B(t,(u+1)x - ut)
\end{equation}
satisfy (\ref{SVM})
for $t \in [0,T], x, v \in \bfR$ with initial data $$f^u_0(x,v) = f_0((u+1)x,(1+u)v - u)$$ and $$B^u_0(x) = (1+u)^{-1} B_0((u+1)x)$$ for every $u  \neq -1$.
\end{lemma}
\begin{proof}
Clearly, if $(f^u,B^u)$ satisfies these properties then choose $u = 0$ and $(f,B)$ will satisfy the same equations.  So, assume that $(f,B)$ solve (\ref{SVM}) and define $(f^u,B^u)$ by (\ref{fu}) and (\ref{Bu}) respectively. Then, a brief calculation shows that $(f^u,B^u)$ satisfy (\ref{SVM}) and the corresponding initial conditions.  Denoting $$t' = t, \qquad x' = (u+1)x - ut \qquad v' = (u+1)v - u $$ we find specifically
$$\begin{gathered}
\partial_t f^u(t,x,v) + v\partial_x f^u(t,x,v) + B^u(t,x) \partial_v f^u(t,x,v)\\ 
= \biggl [ \partial_t f(t',x',v') - u \partial_x f(t',x',v') \biggr ] + v(u+1) \partial_x f(t',x',v')\\
\quad + (1+u)^{-1} B(t',x') (1+u) \partial_v f(t',x',v')\\
= \partial_t f(t',x',v') + v'\partial_x f(t',x',v') + B(t',x') \partial_v f(t',x',v') = 0
\end{gathered}$$
and 
\begin{eqnarray*}
\partial_t B^u(t,x) + \partial_x B^u(t,x) 
& = & (u+1)^{-1} \biggl [ \partial_t B(t',x')  - u\partial_x B(t',x') \biggr ]\\
& & \quad + (u+1)^{-1} \cdot (u+1) \partial_x B(t',x') \\
& = & (u+1)^{-1} \left ( \partial_t B(t',x') + \partial_x B(t',x') \right ) \\
& = & (1+u)^{-1} \int f(t',x',v) \ dv \\
& = & \int f(t',x',v') \ dv' \\
& = & \int f^u(t,x,v) \ dv
\end{eqnarray*}
since $(f,B)$ satisfy (\ref{SVM}) at every point $t \in [0,T]$ and $x,v \in \bfR$.  Similarly the initial conditions are satisfied by inspection.
\end{proof}

\begin{theorem}
\label{T2}
Let $f_0 \in C^1_c(\bfR^2)$ be nonnegative and $B_0 \in C^1(\bfR)$ be nonpositive with the velocity support of $f_0$ contained in $(-\infty, 1)$.  Then, for all $T > 0$ there exist $f \in C^1([0,T] \times \bfR^2)$ and $B \in C^1([0,T]\times\bfR)$ satisfying (\ref{SVM}) with (\ref{IC}) and $f(t,\cdot,\cdot)$ compactly supported for every $t \in [0,T]$.
\end{theorem}
\begin{proof}
We assume $f_0(x,v) \geq 0$ for $x,v \in \bfR$ with $\supp(f(x,\cdot)) \subset (-\infty, 1)$ for every $x \in \bfR$ and $B_0(x) \leq 0$ for $x \in \bfR$.
Let $f$ and $B$ be the local solution guaranteed by Theorem \ref{TMain} and $T > 0$ be given.  By Lemma \ref{L2}, there is a solution of (\ref{SVM}) denoted $(f^u,B^u)$ with $u=-2$ and initial data
$$\begin{gathered}
f_0^u(x,v) = f_0(-x, 2-v)\\
B_0^u(x) = -B_0(-x)
\end{gathered}$$
Hence, this solution satisfies $f_0(x,v), B_0^u(x) \geq 0$ for $x,v \in \bfR$ and $\supp(f_0^u(x,\cdot)) \subset (1, \infty)$.  Thus, by Theorem (\ref{T1}) the solution exists on $[0,T]$ and so must the solution $(f,B)$ launched by $f_0$ and $B_0$.  Since $T > 0$ is arbitrary, the result follows.

\end{proof}


\section{Additional field regularity}
Though we cannot currently show global existence of solutions to (\ref{SVM}) for arbitrary initial data, we now provide an estimate which yields additional regularity of the field $B$.  If one could show \emph{a priori} for any $T > 0$ that $B \in C^1([0,T] \times \bfR)$ then global existence would follow as this bound would imply smoothness of characteristics and $f \in C^1([0,T] \times \bfR^2)$.  Hence, the local solution could be continued up to the arbitrary existence time $T$.  Instead of the result that $B$ gains a full derivative, however, we are able to show that it gains half of a derivative in space-time.

\begin{theorem}
\label{T3}
Let $f_0 \in C_c^1(\bfR^2)$ and $B_0 \in C^1(\bfR)$ be given and let $f \in C^1([0,T^*] \times \bfR^2)$ and $B \in C^1([0,T^*] \times \bfR)$ be the unique solution of (\ref{SVM}) with $T^*$ within the maximal interval of existence.  Then, for any $T > 0$ it follows that $B \in C^{0,1/2}([0,T] \times \bfR)$.
\end{theorem}
\begin{proof}
Let $T > 0$ be given.  We will prove the result  for the gain of regularity in $x$, while a similar argument leads to the additional $1/2$ derivative in $t$.  Let $h > 0$ be given. As in Theorem \ref{T1} we wish to use the operators $$\begin{gathered} \mcS = \partial_t + v \partial_x\\ \mcT = \partial_t + \partial_x. \end{gathered}$$ 
As before, for $v \neq 1$ we may write the inverse transformation $$\partial_x = \frac{1}{1-v} \biggl ( \mcT - \mcS \biggr ).$$  For $t \in [0,T]$, let $$P(t) = \sup \{ \vert v \vert : \exists x \in \bfR \mathrm{ \ with \ } f(t,x,v) \neq 0 \} \leq C_T. $$ 
Since the $(\partial_t, \partial_x) \mapsto (\mcS,\mcT)$ transformation is only valid for $v$ bounded away from one, we decompose the $v$-integral in (\ref{DBDf}) over $[-P(t), P(t)]$ into integrals over the disjoint sets
$$A^\epsilon = \{ v : \vert 1 - v \vert < \epsilon\} \cap [-P(t), P(t)]$$ and $$B^\epsilon = \{ v : \vert 1 - v \vert > \epsilon\} \cap [-P(t), P(t)]$$ where $\epsilon > 0$ is to be chosen. Then, beginning with the representation for the field (\ref{Bfinal}) we have
\begin{eqnarray*}
\vert B(t,x+h) - B(t,x) \vert & \leq & \vert B_0(x+h - t) - B_0(x-t) \vert + \\ & & \left \vert \int_0^t \int \biggl [f(\tau, x + h - t + \tau, v) - f(\tau, x- t + \tau, v) \biggr ] \ dv \ d\tau \right \vert \\
& \leq &  I + II + III
\end{eqnarray*}
where $$I = \Vert B_0 \Vert_{C^{0,1/2}} \cdot \sqrt{h},$$ 
$$II = \int_0^t \int_{A^\epsilon} \vert f(\tau, x + h - t + \tau, v) - f(\tau, x- t + \tau, v) \vert \ dv \ d\tau,$$
and $$ III = \left \vert \int_0^t \int_{B^\epsilon} f(\tau, x + h - t + \tau, v) - f(\tau, x- t + \tau, v) \ dv \ d\tau \right \vert .$$
Using the bound on $f$ we estimate $II$ and find $$II \leq 2t\Vert f_0 \Vert_\infty \cdot \mu( A^\epsilon) \leq C\epsilon $$ where $\mu$ denotes the Lebesgue measure on $\bfR$. As all velocities in $B^\epsilon$ are bounded away from $1$, we can then use the transformation of derivatives for the estimate of $III$.  We include the $x$-derivative of the density and transform $\partial_x$ into terms involving $\mcS$ and $\mcT$ so that
\begin{eqnarray*}
III & = & \left \vert \int_0^t \int_{B^\epsilon} \int_{x-t+\tau}^{x + h - t + \tau} \partial_x f(\tau, y, v) \ dy \ dv \ d\tau \right \vert \\
& = & \left \vert \int_0^t \int_{B^\epsilon} \int_{x-t+\tau}^{x + h - t + \tau} \frac{1}{1-v} \biggl (\mcT f - \mcS f \biggr) (\tau, y, v) \ dy \ dv \ d\tau \right \vert \\
& \leq & III_\mcT + III_\mcS,
\end{eqnarray*}
where the quantities $$III_\mcT = \left \vert \int_0^t \int_{B^\epsilon} \int_{x-t+\tau}^{x + h - t + \tau} \frac{1}{1-v} (\partial_t f + \partial_x f) (\tau, y, v) \ dy \ dv \ d\tau \right \vert$$ and $$III_\mcS = \left \vert \int_0^t \int_{B^\epsilon} \int_{x-t+\tau}^{x + h - t + \tau} \frac{1}{1-v} ( \partial_t f + v\partial_x f) (\tau, y, v) \ dy \ dv \ d\tau \right \vert$$ separate the $\mcT$ and $\mcS$ terms.  To estimate $III_\mcT$ we change variables in the $y$ integral with $z = y - (x - t + \tau)$, switch the order of integration, and integrate by parts in $\tau$ to find
\begin{eqnarray*}
III_\mcT & = &  \left \vert \int_0^t \int_{B^\epsilon} \int_0^h \frac{1}{1-v} \frac{d}{d\tau} f (\tau, z+ x- t + \tau, v) \ dz \ dv \ d\tau \right \vert\\
& = & \left \vert \int_{B^\epsilon} \int_0^h \frac{1}{1-v} ( f (t, z+ x, v) - f(0,z + x - t,v) \ dz \ dv \right \vert\\
& \leq & 2 \Vert f_0 \Vert_\infty \int_{B^\epsilon} \int_0^h \frac{1}{\vert 1-v \vert} \ dz \ dv\\
& \leq & Ch \int_{B^\epsilon}\vert 1 - v \vert^{-1} \ dv
\end{eqnarray*}
Notice that on the set $B^\epsilon$, we have $\vert 1 - v \vert > \epsilon$ and thus $\vert 1 - v \vert^{-1} < \epsilon^{-1}$. In addition, $\mu(B^\epsilon) \leq 2 P(t) \leq C_T$.  Therefore, $$III_T \leq \frac{C_Th}{\epsilon}.$$

Estimating $III_\mcS$, we again use the Vlasov equation of (\ref{SVM}) to write the integrand as a pure $v$-derivative so that $$ \mcS f(\tau, y, v) =(\partial_t f + v \partial_x  f)(\tau,y,v) = - \partial_v ( Bf)(\tau,y,v).$$ Then, we integrate by parts in $v$ and use the bounds on $f$ and $B$, yielding
\begin{eqnarray*}
III_\mcS & \leq & \int_0^t \int_{x - t + \tau}^{x + h - t + \tau} \left \vert \int_{B^\epsilon} \frac{1}{1-v} \partial_v (Bf)(\tau, y, v) \ dv \right \vert \ dy \ d\tau\\
& \leq & C_T \int_0^t \int_{x - t + \tau}^{x + h - t + \tau} \left \vert \left (\frac{Bf}{1-v} \right ) (\tau,y,1 \pm \epsilon) - \int_{B^\epsilon} \frac{Bf}{(1-v)^2}(\tau, y, v) \ dv \right \vert \ dy \ d\tau\\
& \leq & C_T\int_0^t \int_{x - t + \tau}^{x + h - t + \tau} \left ( \frac{1}{\epsilon} +  \int_{B^\epsilon} \frac{1}{\vert1-v \vert^2} \ dv \right ) \ dy \ d\tau\\
& \leq & \frac{C_Th}{\epsilon}.
\end{eqnarray*}
Finally, we combine the estimates to find
\begin{equation}
\label{BHolder}
\vert B(t,x+h) - B(t,x) \vert \leq C_T\biggl ( \sqrt{h} + \epsilon + h/\epsilon \biggr).
\end{equation}
We choose $\epsilon = \sqrt{h}$ optimally so that $\epsilon = h/\epsilon$, and finally $$ \vert B(t,x+h) - B(t,x) \vert  \leq C_T \sqrt{h}.$$  Thus, $B(t,\cdot) \in C^{0,1/2}(\bfR)$ for every $t \in [0,T]$ and a similar argument can be used to establish the H\"{o}lder continuity in $t$ for fixed $x \in \bfR$.
\end{proof}

\section{Acknowledgements} The third author would like to thank the Isaac Newton Institute at Cambridge University for hosting him during the program ``Partial Differential Equations in Kinetic Theories'' and providing an amazing environment to promote collaboration and the exchange of mathematical ideas, including many helpful discussions with S. Calogero.

\bibliographystyle{acm}
\bibliography{SDPrefs}

\end{document}